\documentclass[final]{siamltex}

\usepackage{amssymb}
\usepackage{color}
\usepackage{epsfig}
\usepackage{graphicx}
\usepackage{makeidx}
\usepackage{multicol}
\usepackage{enumerate}
\usepackage{amsfonts}
\usepackage{amsmath}
\usepackage{amsbsy}

\usepackage{epsfig}
\usepackage{graphicx}

 \usepackage[lined]{algorithm2e}
\usepackage{algpseudocode}

\newcommand{\sym}[1]{\mbox{{\bf sym}($#1$)}}

\newdimen\pHeight
\newdimen\pLower
\newdimen\pLineWidth
\newdimen\pKern
\newdimen\pIR
\newsavebox{\Cbox}
\newsavebox{\vertCmplx}
\newdimen\Cheight
\newdimen\Cwidth
\sbox{\Cbox}{\rm C}
\sbox{\vertCmplx}{\rule[\pLower]{\pLineWidth}{\Cheight}}
\sbox{\Cbox}{\usebox{\Cbox}\kern\pKern\usebox{\vertCmplx}}
\wd\Cbox=\Cwidth

\def\R{{\rm I\kern\pIR R}}

\mathcode `:="003A
\newcommand{\inv}[2]{\mbox{${#1} \in
\R^{\hspace*{-.01in} #2}$}}  

\newcommand{\inm}[3]
   {\mbox{${#1} \in \R^{\hspace*{-.015in}{#2} \times
{#3}}\hspace*{-.05in}$ }}

\newcommand{\norm}[1]{\mbox{$\|\: #1 \: \|$}}

\newcommand{\T}{\raisebox{1pt}{$ \:\otimes \:$}}

\mathcode `:="003A

\newenvironment{code}{\begin{tabbing}\hspace*{.3in} \=
\hspace*{.3in} \= \hspace*{.3in} \= \hspace*{.3in} \=
\hspace*{.3in} \= \hspace*{.3in} \= \kill }{\end{tabbing}}

\definecolor{color17}{rgb}{.8,0.0,1}
\definecolor{dark_green}{rgb}{0.0, 0.6, 0.0}
\definecolor{orange}{rgb}{1,0.5,0}
\definecolor{brown}{cmyk} {0, 0.8, 1, 0.6}
\definecolor{violet}{rgb}{0.54, 0.17, 0.89}
\newcommand{\smallE}{\mbox{$\scriptscriptstyle E$}}
\newcommand{\smallN}{\mbox{$\scriptscriptstyle N$}}
\newcommand{\smallPlus}{\mbox{$\scriptscriptstyle +$}}
\newcommand{\smallMinus}{\mbox{$\scriptscriptstyle -$}}
\renewcommand{\sym}{\mbox{\tiny sym}}
\renewcommand{\skew}{\mbox{\tiny skew}}

\newcommand{\unfold}[4]{\mbox{\tiny $[#1,\!#2]\!\!\times\! \![#3,\!#4]$}}
\newcommand{\nn}{\mbox{\tiny $nn$}}
\newcommand{\symn}{\mbox{$\mbox{\tt sym}_{\hspace*{.8pt}\mbox{\footnotesize  \tt  n}}$}}
\newcommand{\skewn}{\mbox{$\mbox{\tt skew}_{\hspace*{.8pt}\mbox{\footnotesize  \tt  n}}$}}

\title{Approximating Matrices with Multiple Symmetries}

\author{Charles F. Van Loan\thanks{Department of Computer Science, Cornell University,
        Ithaca, NY 14853, {\tt cv@cs.cornell.edu}. }
\and
Joseph P. Vokt\thanks{Department of Computer Science, Cornell University, 
        Ithaca, NY 14853, {\tt jpv52@cornell.edu}.}
}

\begin{document}
\maketitle
\begin{abstract}
If a tensor with various symmetries is properly unfolded, then the resulting matrix inherits those symmetries.
As tensor computations become increasingly important it is imperative that we develop efficient structure preserving methods for matrices with multiple symmetries. In this paper we consider how to exploit and preserve structure in the pivoted Cholesky factorization when approximating a matrix $A$ that is both symmetric ($A=A^T$) and what we call {\em perfect shuffle symmetric}, or {\em perf-symmetric}. The latter property means that $A = \Pi A\Pi$ where
$\Pi$ is a permutation with the property that $\Pi v = v$ if $v$ is the vec of a symmetric matrix and
$\Pi v = -v$ if $v$ is the vec of a skew-symmetric matrix. Matrices with this structure can arise when an order-4
tensor $\cal A$ is unfolded and its elements satisfy ${\cal A}(i_{1},i_{2},i_{3},i_{4}) = 
{\cal A}(i_{2},i_{1},i_{3},i_{4}) ={\cal A}(i_{1},i_{2},i_{4},i_{3}) ={\cal A}(i_{3},i_{4},i_{1},i_{2}).$ This is the
case in certain quantum chemistry applications where the tensor entries are electronic repulsion integrals.
Our technique involves a closed-form block diagonalization followed by one or two half-sized pivoted
Cholesky factorizations. This framework allows for a lazy evaluation feature that is important if the entries in $\cal A$ are expensive to compute. In addition to being a structure preserving rank reduction technique, we find that this approach for obtaining the Cholesky factorization reduces the work by up to a factor of 4.

\end{abstract}


% 15A18 : Eigenvalues, singular values, and eigenvectors
% 15A69:ÊMultilinear algebra, tensor products
% 65F30 : Other matrix algorithms
% 65F15 : Eigenvalues, eigenvectors (numerical linear algebra)
\pagestyle{myheadings}
\thispagestyle{plain}
\markboth{C.F. VAN LOAN AND J.P. VOKT}{Approximating Matrices with Multiple Symmetries }

\section{Introduction}

Low-rank approximation and the exploitation of structure are important themes throughout  matrix computations.
This paper revolves around some basic tensor calculations that reinforce this point. The tensors involved have
multiple symmetries and the same can be said of the matrices that arise if they are obtained by a suitable unfolding. 
\subsection{Motivation}
Our contribution is prompted by the following problem.
Suppose  
${\cal A}\in \R^{n \times n \times n \times n} $
is an order-4 tensor
with the property that its entries satisfy
\begin{equation}
\mathcal{A}(i_{1},i_{2},i_{3},i_{4})  \:=\: 
\left\{
\begin{array}{c} 
\mathcal{A}(i_{2},i_{1},i_{3},i_{4})\\
 \mathcal{A}(i_{1},i_{2},i_{4},i_{3}) \rule{0pt}{14pt}\\
 \mathcal{A}(i_{3},i_{4},i_{1},i_{2}) \rule{0pt}{14pt}
\end{array}.
\right.
\end{equation}
We say that such a tensor is ((1,2),(3,4))-{\em symmetric}.  See {\sc Fig} 1.1 for 
an $n=3$ example.
Given \inm{X}{n}{n} the challenge is to
compute efficiently the tensor
${\cal B} \in \R^{n \times n \times n \times n}\rule{0pt}{12pt}$ defined by
\begin{equation}
 {\cal B}(i_{1},i_{2},i_{3},i_{4}) \;=\!
 \sum_{j_{1},j_{2},j_{3},j_{4}=1}^{n}  \!\!
{\cal A}(j_{1},j_{2},j_{3},j_{4})
X(i_{1},j_{1}) X(i_{2},j_{2})X(i_{3},j_{3})X(i_{4},j_{4}).
\end{equation}
This is a highly structured {\em multilinear product}.   As with many tensor computations, (1.2) can be 
reformulated as a matrix computation. In particular, it can be shown that
\begin{equation}
B = (X \T X)A(X \T X)^{T}
\end{equation}
where $A$ and $B$ are $n^{2}$-by-$n^{2}$ matrices that are obtained by certain unfoldings of  the tensors
$\cal A$ and $\cal B$. Depending upon the chosen unfolding, the matrices $A$ and $B$  inherit the  tensor symmetries (1.1).

\begin{figure}
\begin{center}
\footnotesize
\begin{tabular}{|c|cccccccc|}
\hline
\footnotesize Value &\multicolumn{8}{|c|}{\footnotesize Entries that Share that Value $\rule[-4pt]{0pt}{14pt}$}\\
\hline
1 & (1,1,1,1) & & & & & & &  $\rule{0pt}{8pt}$ \\
2 & (2,1,1,1) & (1,2,1,1) & (1,1,2,1) & (1,1,1,2) & & & & \\
3 & (3,1,1,1) & (1,3,1,1) & (1,1,3,1) & (1,1,1,3) & & & &\\
4 & (2,2,1,1) & (1,1,2,2) & & & & & &  \\
5 & (3,2,1,1) & (2,3,1,1) & (1,1,3,2) & (1,1,2,3) & & & &\\
6 & (3,3,1,1) & (1,1,3,3) & & & & & & \\
7 & (2,1,2,1) & (1,2,2,1) & (2,1,1,2) &  (1,2,1,2) & & & & \\
8& (1,3,2,1) & (3,1,2,1) & (1,3,1,2) & (3,1,1,2) & (2,1,1,3) & (1,2,1,3) & (2,1,3,1) & (1,2,3,1)\\
9 & (2,2,2,1) & (2,2,1,2) & (2,1,2,2) & (1,2,2,2)& & & &\\
10 & (3,2,2,1) & (2,3,2,1) & (3,2,1,2) & (2,3,1,2) & (2,1,3,2) & (1,2,3,2) & (2,1,2,3) & (1,2,2,3)    \\
11 & (3,3,2,1) & (3,3,1,2) & (2,1,3,3) & (1,2,3,3) & & & &\\
12 & (3,1,3,1) & (1,3,3,1) & (3,1,1,3) & (1,3,1,3) & & & &\\
13 & (2,2,3,1) & (2,2,1,3) & (3,1,2,2) & (1,3,2,2) & & & &\\
14 & (3,2,3,1) & (2,3,3,1) & (3,2,1,3) & (2,3,1,3) & (3,1,3,2) & (1,3,3,2) & (3,1,2,3) & (1,3,2,3) \\
15 & (3,3,3,2) & (3,3,2,3) & (3,2,3,3) & (2,3,3,3) & & & &\\
16 & (2,2,2,2) & & & & & & &\\
17 & (3,2,2,2) & (2,3,2,2) & (2,2,3,2) & (2,2,2,3) & & & &\\
18 & (3,3,2,2) & (2,2,3,3) & & & & & &  \\
19 & (3,2,3,2) & (2,3,3,2) & (3,2,2,3) & (2,3,2,3)& & & & \\
20 & (3,2,3,3) & (2,3,3,3) & (3,3,3,2) & (3,3,2,3) & & & &\\
21 & (3,3,3,3) & & & & & & & $\rule[-4pt]{0pt}{6pt}$\\
\hline
\end{tabular}
\end{center}
\caption{An example of a \mbox{\rm ((1,2),(3,4))}-symmetric tensor 
\mbox{\rm ($n = 3$)}. It has at most 21 distinct values. Equations (1.5) and (1.7) show what this tensor  looks like when unfolded into a $9\times 9$ matrix. In general, the subspace of $\R^{n\times n \times n \times n}$
defined by  all  \mbox{\rm ((1,2),(3,4))}-symmetric tensors
has dimension $(n^4 + 2n^{3} + 3n^{2} +2n)/8$.
}
\end{figure}
For example, suppose  $A = {\cal A}_{\unfold{1}{3}{2}{4}}$
is the ``$[1,3]\times[2,4]$ unfolding''  defined by

\medskip

\begin{equation}
{\cal A}(i_{1},i_{2},i_{3},i_{4}) \; \rightarrow \; A(i_{1}+(i_{3}-1)n,i_{2}+(i_{4}-1)n).
\end{equation}

\medskip

\noindent
This $n^{2}$-by-$n^{2}$ matrix can be regarded as $n$-by-$n$ block matrix $A = (A_{pq})$ whose blocks $A_{pq}$ 
are $n$-by-$n$ matrices.
It follows from (1.4) that 
\[
{\cal A}(i_{1},i_{2},i_{3},i_{4}) = [A_{i_{3},i_{4}}]_{i_{1},i_{2}}.
\]
Combining this with (1.1)  we conclude that $A_{qp} = A_{pq} = A_{pq}^{T}$. Note that this implies $A^{T}=A$.
 To visualize the structure associated with the $[1,3]\times[2,4]$ unfolding, suppose that ${\cal A} \in \R^{3\times 3\times 3 \times 3}$ is defined by {\sc Fig} 1.1. It follows that

\small
\begin{equation}
{\cal A}_{\unfold{1}{3}{2}{4}}  \:=\: 
\left[ \begin{array}{ccc|ccc|ccc}
 1 & 2 & {3}  &2 & 7  &{8} & 3  &8  &{12} \\
 2 & 4 & 5  & 17 & 19 &  10 & 8 & 13 & 14  \rule{0pt}{11pt}\\
 3 & 5 & 6 & 8  & 10 &  11 & 12 & 14 & 15  \rule{0pt}{11pt}\\
\hline
 2 & 7 & {8} & 4  & 9 & {13} & 5 & 10 & {14} \rule{0pt}{11pt} \\
 7 &9 & 10 & 9 &  16  & 17 & 10 & 17 &  19 \rule{0pt}{11pt}\\
 8 & 10 & 11 & 13 & 17 & 18 & 14 & 19 & 20 \rule{0pt}{11pt}\\
\hline
 3 & 8 & {12} & 5 & 10  & {14} & 6 & 11 & {15} \rule{0pt}{11pt}\\
 8 & 13 & 14 & 10 & 17 & 19 & 11 & 18 & 20 \rule{0pt}{11pt}\\
 12 & 14 & 15 & 14 & 19 & 20 & 15 & 20 & 21 \rule{0pt}{11pt}
\end{array}\right].
\end{equation}
\normalsize

\bigskip

\noindent
On the other hand, the $[\:1,2\:]\!\times\! [\:3,4\:]$ unfolding $A = {\cal A}_{\unfold{1}{2}{3}{4}}$
defined by
\begin{equation}
{\cal A}(i_{1},i_{2},i_{3},i_{4}) \; \rightarrow \; A(i_{1}+(i_{2}-1)n,i_{3}+(i_{4}-1)n)
\end{equation}

\noindent
results in a matrix $A$ with different properties. Indeed, if we apply this mapping to the tensor defined in {\sc Fig} 1.1,  then we obtain

\medskip

\small
\begin{equation}
{\cal A}_{\unfold{1}{2}{3}{4}}  \:=\:
\left[ \begin{array}{ccc|ccc|ccc}
 1 & 2 & 3 & 2 & 4 & 5 & 3 & 5 & 6  \\
 2 & 7 & 8 & 7 & 9 & 10 & 8 & 10 & 11 \rule{0pt}{11pt} \\
 3 & 8 & 12 & 8 & 13 & 14 & 12 & 14 & 15 \rule{0pt}{11pt} \\
\hline
 2 & 7 & 8 & 7 & 9 & 10 & 8 & 10 & 11  \rule{0pt}{11pt}\\
 4 & 9 & 13 & 9 & 16 & 17 & 13 & 17 & 18 \rule{0pt}{11pt} \\
 5 & 10 & 14 & 10 & 17 & 19 & 14 & 19 & 20  \rule{0pt}{11pt}\\
\hline
 3 & 8 & 12 & 8 & 13 & 14 & 12 & 14 & 15  \rule{0pt}{11pt}\\
 5 & 10 & 14 & 10 & 17 & 19 & 14 & 19 & 20  \rule{0pt}{11pt}\\
 6 & 11 & 15 & 11 & 18 & 20 & 15 & 20 & 21 \rule{0pt}{11pt}
\end{array}\right].
\end{equation}

\normalsize

\medskip

\noindent
It is easy to prove that this unfolding is also symmetric. (Just combine (1.6) with the observation that 
 ${\cal A}(i_{1},i_{2},i_{3},i_{4}) = {\cal A}(i_{3},i_{4},i_{1}, i_{2})$.) 
But it also satisfies a type of symmetry that is related to a particular  perfect shuffle permutation.
To see this we define the $n^{2}$-by-$n^{2}$ permutation matrix $\Pi_{\nn}$ by
\begin{equation}
\Pi_{\nn} \:=\: I_{n^{2}}(:,p),\qquad
p \;=\; [ \: 1:n:n^2 \;|\;  2:n:n^2 \;|\; \cdots \;|\; n:n:n^2 \:]
\end{equation}
where we are making use of the {\sc Matlab} colon notation. Here is an example:
\begin{equation}
\Pi_{33} \:=\:
\left[
\begin{array}{ccc|ccc|ccc}
1 & 0 & 0 & 0 & 0 & 0 & 0 & 0 & 0 \\
0 & 0 & 0 & 1 & 0 & 0 & 0 & 0 & 0 \rule{0pt}{11pt} \\
0 & 0 & 0 & 0 & 0 & 0 & 1 & 0 & 0 \rule{0pt}{11pt} \\  \hline
0 & 1 & 0 & 0 & 0 & 0 & 0 & 0 & 0 \rule{0pt}{11pt} \\
0 & 0 & 0 & 0 & 1 & 0 & 0 & 0 & 0 \rule{0pt}{11pt} \\
0 & 0 & 0 & 0 & 0 & 0 & 0 & 1 & 0 \rule{0pt}{11pt} \\ \hline
0 & 0 & 1 & 0 & 0 & 0 & 0 & 0 & 0 \rule{0pt}{11pt} \\
0 & 0 & 0 & 0 & 0 & 1 & 0 & 0 & 0 \rule{0pt}{11pt} \\
0 & 0 & 0 & 0 & 0 & 0 & 0 & 0 & 1 \rule{0pt}{11pt}
\end{array}
\right] \:=\: I_{9}(\,:\,,[\:1\:4\:7\:2\:5\:8\:3\:6\:9\:]).
\end{equation}
A matrix \inm{A}{n^{2}}{n^{2}} is {\em perfect shuffle invariant} if
\begin{equation}
A = \Pi_{\nn} A\Pi_{\nn}
\end{equation}
and
{\em PS-symmetric} 
if it is both symmetric and perfect shuffle invariant. 

In \S2 we show how to construct a reduced rank
approximation to a PS-symmetric matrix that is also PS-symmetric.
This is important in the evaluation 
of the multilinear product (1.2). In particular, it  enables us to approximate the
unfolding matrix $A$ in (1.4) with a relatively short sum of structured Kronecker products:
\begin{equation}
A \:\approx \: \sum_{i=1}^{r} \sigma_{i} \cdot  C_{i} \T C_{i}\qquad \inm{C_{i}^{T}=C_{i}}{n}{n},\;r<<n^{2}.
\end{equation}
It then follows from (1.3) that
\begin{equation}
B \:\approx \: \sum_{i=1}^{r} \sigma_{i} \cdot (X\T X)( C_{i} \T C_{i})(X \T X)^{T} \:=\: \sum_{i=1}^{r} \sigma_{i} (XC_{i}X^{T}) \T (XC_{i}X^{T}).
\end{equation}
The $\{\sigma_{i},C_{i},X\}$ representation of $B$ (and hence ${\cal B}$) is an $O(rn^{3})$ computation. 

The expansion (1.11) looks like a Kronecker-product SVD of $A$ \cite{gvl4, cvl1, cvl2}. However, the  method that we propose in this paper is not based on expensive svd-like computations 
but  on a structured factorization that 
combines  block diagonalization with a pair of ``half-size'' pivoted Cholesky factorizations.
Recall that if \inm{M}{\smallN}{\smallN} is symmetric and positive semidefinite with $r\leq N$ positive eigenvalues, then
the pivoted Cholesky factorization (in exact arithmetic) computes  the factorization
\[
PMP^{T} \:=\: LL^{T} 
\]
where \inm{P}{\smallN}{\smallN} is a permutation matrix and  \inm{L}{\smallN}{r}  is  lower triangular with $r = \mbox{rank}(A)$, 
It follows that if
$Y = P^{T}L = [ y_{1},\ldots,y_{r}]$ is a column partitioning,  then
\[
M \:=\: (P^{T}L)(P^{T}L)^{T} \:=\: YY^{T} \:=\: \sum_{k=1}^{r} y_{k}y_{k}^{T}.
\]
In practice, $r$ is (numerically) determined during the factorization process. More on this in \S4.2.
We mention that this rank-$r$ representation requires $Nr^{2}-2r^{2}/3 + O(Nr)$ flops. See \cite[pp.165-66]{gvl4} for more details\footnote{
The connection between the pivoted LDL factorization $PMP^{T} = \tilde{L}D\tilde{L}^{T}$ where $\tilde{L}$ is unit lower triangular and the pivoted Cholesky factorization $PMP^{T} = LL^{T}$ is simple. The lower triangular Cholesky factor $L$ is given by $L = \tilde{L}\cdot \mbox{diag}(d_{1}^{1/2},\ldots, d_{r}^{1/2})$.
Virtually all of the rank-revealing operations in this paper can be framed in ``LDL'' language. We use the Cholesky representation
so that readers can more easily relate our work to what has already been published and to existing procedures in LAPACK.
}.

\subsection{Overview of the Paper}

In \S2 we discuss the properties of PS-symmetric matrices. A key result is the derivation of
a simple orthogonal matrix $Q$ that can be used to block-diagonalize a PS-symmetric matrix $A$:
$Q^{T}AQ = \mbox{diag}(A_{1},A_{2})$. Rank-revealing  pivoted Cholesky factorizations are then applied to
the half-sized diagonal blocks. The resulting factor matrices are then combined with $Q$ to produce a
rank-1 expansion for $A$ with terms that are also PS-symmetric. In \S3 we apply these results to
compute a structured multilinear product whose defining tensor $\cal A$  is ((1,2),(3,4))-symmetric.
An application from quantum chemistry is considered that has a dramatic low-rank feature.
Implementation details and benchmarks are 
provided in \S4. Anticipated future work and conclusions are offered in \S5.

\subsection{Centrosymmetry: An Instructive Preview}

We conclude the introduction with a brief discussion of matrices that
are centrosymmetric. These are matrices  that are symmetric about  their diagonal \underline{and} antidiagonal, e.g.,
\[
A \:=\:
\left[
\begin{array}{cccc}
a & b & c& d\\
b & e & f & c \\
c & f & e &  b\\
d & c & b & a
\end{array}
\right].
\]
They are a particularly simple class of multi-symmetric matrices and because of that they can be used  to
anticipate  the
main ideas that follow in  \S2 and \S3. For a more in-depth treatment of centrosymmetry, see
Andrew \cite{andrew}, Datta and Morgera \cite{datta} and Pressman \cite{press}.

Formally, a matrix \inm{A}{n}{n} is {\em centrosymmetric} if  $A = A^{T}$ and $A = E_{n}AE_{n}$ where
\[
\inm{E_{n}=I_{n}(:,n:{-1}:1)}{n}{n}
\]
is
the $n$-by-$n$ {\em exchange permutation}.
The redundancies among the elements of a centrosymmetric matrix are nicely exposed through blocking.
Assume for clarity that $n = 2m$. (The odd-$n$ case is basically the same.) If
\[
A \:=\: \left[ \begin{array}{cc} A_{11} & A_{12} \\ A_{21} & A_{22} \rule{0pt}{14pt} \end{array} \right]
\qquad  \inm{A_{ij}}{m}{m}
\]
is centrosymmetric, then by substituting
\[
E_{n} \:=\: \left[ \begin{array}{cc} 0 & E_{m} \\ E_{m} & 0 \rule{0pt}{15pt}\end{array} \right]
\]
into the equation $A = E_{n}AE_{n}$ we see that $A_{21} = E_{m}A_{12}E_{m}$ and
$A_{22} = E_{m}A_{11}E_{m}$, i.e.,
\begin{equation}
 A \:= \: \left[ \begin{array}{cc} A_{11} & A_{12} \\ E_{m}A_{12}E_{m} & E_{m}A_{11}E_{m}\rule{0pt}{15pt}
\end{array} \right].
\end{equation}
Moreover,  $A_{11}$ and $A_{12}E_{m}$ are each symmetric.
Given this block structure it is easy to confirm that the orthogonal matrix
\begin{equation}
Q_{\smallE} \:=\: \frac{1}{\sqrt{2}}
\left[ \begin{array}{c|c} I_{m} & I_{m} \\ E_{m} & -E_{m} \rule{0pt}{15pt} \end{array} \right] \:\equiv \;
\left[ \begin{array}{c|c} Q_{\smallPlus} & Q_{\smallMinus}  \end{array} \right]
\end{equation}
block diagonalizes $A$:
\begin{equation}
Q_{\smallE}^{T} A Q_{\smallE} \:=\: \left[ \begin{array}{cc} A_{11}+A_{12}E_{m} & 0 \\ 0 & A_{11} - A_{12}E_{m} \rule{0pt}{15pt} \end{array} \right]\:\equiv \:
\left[ \begin{array}{cc} A_{\smallPlus} & 0 \\ 0 & A_{\smallMinus}
\rule{0pt}{15pt} \end{array} \right].
\end{equation}
If $A$ is positive semidefinite, then the same can be said of 
$A_{\smallPlus}$ and $A_{\smallMinus}$ and we can compute the
following half-sized pivoted Cholesky factorizations:
\begin{eqnarray}
P_{\smallPlus}A_{\smallPlus}P_{\smallPlus}^{T} &=& L_{\small+}L_{\smallPlus}^{T}  \\
P_{\smallMinus}A_{\smallMinus}P_{\smallMinus}^{T} &=& L_{\smallMinus}L_{\smallMinus}^{T} .\rule{0pt}{18pt}
\end{eqnarray}

\medskip

\noindent
If we define the matrices \inm{Y_{\smallPlus}}{n}{m} and 
\inm{Y_{\smallMinus}}{n}{m}  by
\begin{eqnarray}
Y_{\smallPlus} &=& Q_{\smallPlus}P_{\smallPlus}^{T}L_{\smallPlus} \:=\: [ \:y^{(1)}_{\smallPlus} \:| \: \cdots \:|\:y_{\smallPlus}^{(m)} \:]\\
Y_{\smallMinus} &=& Q_{\smallMinus}P_{\smallMinus}^{T}L_{\smallMinus} \:=\: [ \:y^{(1)}_{\smallMinus} \:| \: \cdots \:|\:y_{\smallMinus}^{(m)} \:], \rule{0pt}{18pt}
\end{eqnarray}

\medskip

\noindent
then it follows from (1.15)-(1.19) that

\medskip

\[
A \:=\:  Y_{\smallPlus}Y_{\smallPlus}^{T} \:+\: Y_{-}Y_{-}^{T} \;=\;
\sum_{i=1}^{m} y_{\smallPlus}^{(i)}\,[y_{\smallPlus}^{(i)}]^{T}
\:+\: \sum_{i=1}^{m} y_{\smallMinus}^{(i)}\,[y_{\smallMinus}^{(i)}]^{T}.
\]

\medskip

\noindent
Each of the rank-1 matrices in this expansion is
centrosymmetric
because $E_{n}Y_{\smallPlus} = Y_{\smallPlus}$
and $E_{n}Y_{\smallMinus} = -Y_{\smallMinus}$.
It follows that if
$r_{\smallPlus} \leq m$ and $r_{\smallMinus}\leq m$, then
\begin{equation}
A_{\{r_{\smallPlus},r_{\smallMinus}\}} \:=\:
\sum_{i=1}^{r_{\smallPlus}} y_{\smallPlus}^{(i)}\,[y_{\smallPlus}^{(i)}]^{T}
\:+\: \sum_{i=1}^{r_{\smallMinus}} y_{\smallMinus}^{(i)}\,[y_{\smallMinus}^{(i)}]^{T}
\end{equation}
is centrosymmetric and rank$(A_{\{r_{\smallPlus},r_{\smallMinus}\}}) =  r_{\smallPlus}+r_{\smallMinus}$. 
Thus, by combining the block diagonalization (1.15) with the pivoted Cholesky factorizations (1.16)-(1.17) we can approximate
a given positive semidefinite centrosymmetric matrix with a matrix of lower rank that is 
also centrosymmetric. 

We briefly consider the efficiency of such a maneuver keeping in mind 
the ``preview nature'' of this subsection.
Here are some obvious implementation concerns:

\medskip

\begin{enumerate}
  \item What is the cost of the block diagonalization? The matrices 
$A_{\smallPlus}$ and $A_{\smallMinus}$ are simple enough, but is their formation a negligible 
overhead?
\item  From the flop point of view, halving the dimension of an $O(n^{3})$ factorization  reduces the volume of
arithmetic
by a factor of 8. Is the cost of computing the pivoted Cholesky's of $A_{\smallPlus}$ and $A_{\smallMinus}$ 
one-fourth the cost of the  single full-size pivoted Cholesky of $A$? 
\item If $A$ is close to a matrix with very low rank and/or its entries  $a_{ij}$ are expensive to compute, then it may be preferable to work with a left-looking implementation of pivoted Cholesky that computes matrix entries on a ``need to know'' basis. How can one organize pivot determination in such a setting?
\end{enumerate}

\medskip

\noindent
 The table in Fig 1.2 sheds  light on some of these issues by comparing the computation of
 the structured approximation $A_{\{r_{\smallPlus},r_{\smallMinus}\}}$ with the unstructured  approximation $A_{r}$ based on  $PAP^{T} = LL^{T}$, i.e.,
\[
A_{r} \:=\: \sum_{i=1}^{r} y^{(i)}[y^{(i)}]^{T}
\]
where $P^{T}L = [y^{(1)},\ldots,y^{(r)}]$ and  $r = r_{\smallPlus}+r_{\smallMinus}$.
\begin{figure}[h]
\begin{center}
\small

\bigskip

\begin{tabular}{|c||c|c||c|c|}
\hline
& \multicolumn{2}{|c||}{$r_{\smallPlus}=n/2\:,\;r_{\smallMinus} = n/2 \rule[-5pt]{0pt}{14pt}$} &
\multicolumn{2}{|c|}{$r_{\smallPlus}=n/100\:,\;r_{\smallMinus} = n/100$}\\
\hline
$n$ & $T_{u}/T_{s}$ & $T_{\mbox{\tiny \em set-up}}/T_{s}$& $T_{u}/T_{s}$ &$T_{\mbox{\tiny \em set-up}}/T_{s}$ \rule[-5pt]{0pt}{14pt}\\
\hline
1500 & 1.93 & 0.32 & 0.53 & 0.66  \rule[-4pt]{0pt}{14pt}\\
3000 & 2.68 & 0.22 & 0.56 & 0.69  \rule[-4pt]{0pt}{14pt}\\
4500 & 2.89 & 0.17 & 0.83 & 0.64  \rule[-4pt]{0pt}{14pt}\\
6000 & 3.18 & 0.13 & 0.88 & 0.65  \rule[-4pt]{0pt}{14pt}\\
\hline
\end{tabular}
\end{center}
\caption{$T_{u}$ is the time required to compute the Cholesky factorization of $A$, $T_{s}$ is the
time required to set up  $A_{\smallPlus}$ and $A_{\smallMinus}$ and compute their Cholesky factorizations, and $T_{\mbox{\tiny \em set-up}}$ is the time required to just set-up $A_{\smallPlus}$ and $A_{\smallMinus}$. The LAPACK procedures POTRF (unpivoted Cholesky calling level-3 BLAS) and PSTRF (pivoted Cholesky calling level-3 BLAS) were used for full rank and low rank cases respectively. Results are based on running numerous  random trials for each combination of $n$ and $(r_{\smallPlus},r_{\smallMinus})$. A  single core of the Intel(R) Core(TM) i5-3210M CPU @ 2.50GHz was used.}
\end{figure}

In the full rank case ($r_{\smallPlus}=n/2\:,\;r_{\smallMinus} = n/2$), a flop-only analysis would predict a speed-up factor of
4 since we are replacing one $n$-by-$n$ Cholesky factorization ($n^{3}/3$ flops) with a pair of 2 half-size factorizations ($2\cdot (n/2)^3/3$ flops.) The ratios $T_{u}/T_{s}$ are somewhat less than this because there is an $O(n^{2})$ overhead associated with the setting
up of the matrices $A_{\smallPlus}$ and $A_{\smallMinus}$. This is quantified by the ratios $T_{\mbox{\tiny \em set-up}}/T_{s}$.

The low-rank results point to  the importance of having a ``lazy evaluation'' strategy when it comes to setting up the
matrices $A_{\smallPlus}$ and $A_{\smallMinus}$. The LAPACK routines that are used are right-looking and thus require the complete 
$O(n^{2})$ set-up of these half-sized matrices. However, the flop cost of the pivoted Cholesky factorizations of these low rank matrices
is $O(nr^{2})$. Thus, the set-up costs dominate and there is a serious tension between efficiency and  structure preservation.  What we need
is  a {\em left-looking} pivoted Cholesky procedure that involves an $O(nr)$ set-up cost. We shall discuss just such a framework in \S4.2
in the context of a highly structured low-rank PS-symmetric approximation problem.

\section{Perfect Shuffle Symmetry}

Just as centrosymmetry is defined by the equation $A = E_{n}AE_{n}$,  $PS$-symmetry is defined by the equation 
$A = \Pi_{\nn} A \Pi_{\nn}$ where $\Pi_{\nn}$ is a particular perfect shuffle permutation. We start by looking at the eigenvectors of this permutation. This leads to the construction of a simple  orthogonal matrix (like $Q_{\smallE}$
in (1.14))
that can be used to block diagonalize a PS-symmetric matrix. A framework for structured low-rank  approximation
follows.

\subsection{Perfect Shuffle Properties}
Perfect shuffle permutations relate matrix transposition to vector permutation.
Following Van Loan \cite[p.78]{cvFFT}, if $m = pr$, then the perfect shuffle
permutation \inm{\Pi_{pr}}{m}{m} is defined by
\[
\Pi_{pr} \:=\: I_{n}(:,[(1:r:m)\: (2:r:m) \: \cdots \: (r:r:m)]).
\]
The action of $\Pi_{pr}$ is best described 
using the {\sc Matlab} {\tt reshape} operator, e.g.,
\[
\inv{x}{12} \; \Rightarrow \; \mbox{\tt reshape}(x,3,4) \;=\; 
\left[
\begin{array}{cccc}
x_{1} & x_{4} & x_{7} & x_{10} \\
x_{2} & x_{5} & x_{8} & x_{11} \\
x_{3} & x_{6} & x_{9} & x_{12}
\end{array}\right].
\]
If \inv{x}{pr}, then
\begin{equation}
y = \Pi_{pr}x \quad \Rightarrow \quad \mbox{\tt reshape}(y,p,r) \:=\: \mbox{\tt reshape}(x,r,p)^{T}.
\end{equation}
In other words, if \inm{S}{r}{p}, then $\mbox{\tt vec}(S^{T}) = \Pi_{pr}\mbox{\tt vec}(S)$.

We shall be interested in the case $p=r=n$. Using (2.1) it is easy to see 
that  $\Pi_{\nn}\Pi_{\nn} = I$  showing that  $\Pi_{\nn} = \Pi_{\nn}^{T}$. Thus, if
$\lambda$ is an eigenvalue of  $\Pi_{\nn}$,
then $\lambda = 1$ or $\lambda = -1$. Using (2.1) again it follows that
\begin{eqnarray*}
\Pi_{\nn}x = +x & \quad \Rightarrow \quad & S = \mbox{\tt reshape}(x,n,n) \:\mbox{is symmetric}\\
\Pi_{\nn}x = -x & \quad \Rightarrow \quad & S = \mbox{\tt reshape}(x,n,n) \:\mbox{is skew-symmetric} \rule{0pt}{14pt}
\end{eqnarray*}
Thus, $\Pi_{\nn}x=x$ if and only if $S=S^T$. Likewise, $\Pi_{\nn}x=-x$ if and only if $S=-S^T$.

Using these observations about $\Pi_{\nn}$, it is easy to verify that the entries in a PS-symmetric matrix
$A = \Pi_{\nn}A\Pi_{\nn}$ 
 satisfy
\begin{equation}
\begin{array}{cl}
 & A(i_{1}+(i_{2}-1)n,j_{1}+(j_{2}-1)n)\\
=\!& A(j_{1}+(j_{2}-1)n,i_{1}+(i_{2}-1)n) \rule{0pt}{13pt}\\
=\!& A(i_{2}+(i_{1}-1)n,j_{2}+(j_{1}-1)n) \rule{0pt}{13pt}\\
=\!& A(j_{2}+(j_{1}-1)n,i_{2}+(i_{1}-1)n) \rule{0pt}{13pt}
\end{array}
\end{equation}
where it is understood that the indices $i_{1}$, $i_{2}$, $j_{1}$, and $j_{2}$ range from 1 to $n$.

\subsection{Block Diagonalization}

Define the subspaces $\mathbf{S}_{\nn}^{(\sym)}$ and $\mathbf{S}_{\nn}^{(\skew)}$ by
\begin{eqnarray}
\mathbf{S}_{\nn}^{(\sym)} &=& \{\inv{x}{n^{2}} \:|\: \Pi_{\nn}x = x\:\}\\
\mathbf{S}_{\nn}^{(\skew)} &=& \{\inv{x}{n^{2}} \:|\: \Pi_{\nn}x = -x\:\}. \rule{0pt}{15pt}
\end{eqnarray}
It is easy to verify that
$[\mathbf{S}_{\nn}^{(\sym)}]^{\perp} \:=\: \mathbf{S}_{\nn}^{(\skew)}$. Moreover, if
\inm{A}{n^{2}}{n^{2}} is PS-symmetric 
and $x \in \mathbf{S}_{\nn}^{(\sym)}$,
then
\[
Ax \:=\:(\Pi_{\nn}A\Pi_{\nn})x \:=\:
(\Pi_{\nn}A)(\Pi_{\nn}x) \:=\:
(\Pi_{\nn}A)x \:=\: \Pi_{\nn}(Ax)
\]
which shows that  $\mathbf{S}_{\nn}^{(\sym)}$ is an invariant subspace for $A$. 
The subspace $\mathbf{S}_{\nn}^{(\skew)}$ is also
invariant for $A$ by similar reasoning.

Using these facts we can construct a sparse orthogonal matrix $Q_{\nn}$ 
that can be used to block diagonalize a PS-symmetric matrix. Let
 $I_{n} = [e_{1},\ldots,e_{n}]$ be a column partitioning and define
the matrices
\begin{equation}
\begin{array}{lclclcl}
Q_{\nn}^{(\sym)} &=& \left[ 
\begin{array}{c|c|c}
  q_{1}^{(\sym)} & \cdots &q_{n_{\sym}}^{(\sym)}\end{array} \right] & & n_{\mbox{\rm \tiny sym}} &=& n(n+1)/2 \\
Q_{\nn}^{(\skew)} &=& \left[  
\begin{array}{c|c|c} q_{1}^{(\skew)}\! & \cdots & q_{n_{\skew}}^{(\skew)}\!\end{array} \right] & &  n_{\mbox{\rm \tiny skew}} &=& n(n-1)/2 \rule{0pt}{15pt}
\end{array}
\end{equation}
as follows

\begin{code}
\\
\> $k=0$\\
\> {\bf for} $j=1:n\rule{0pt}{12pt}$\\
\>\> {\bf for} $i=j:n\rule{0pt}{12pt}$\\
\>\>\>   $k = k+1 \rule{0pt}{12pt}$\\
\>\>\>    $q_{k}^{(\sym)}    \:=\:\left\{ \begin{array}{ll}
(e_{i}\T e_{j} + e_{j} \T e_{i})/\sqrt{2} & \mbox{if $i>j$}\\
\,e_{i}\T e_{i} \rule{0pt}{15pt} & \mbox{if $i=j$} \end{array}\right.$\\
\>\> {\bf end} $\rule{0pt}{0pt}$\\
\>{\bf end} \`(2.6)\\
\> $k=0\rule{0pt}{12pt}$\\
\> {\bf for} $j=1:n-1 \rule{0pt}{12pt}$\\
\>\> {\bf for} $i=j+1:n\rule{0pt}{12pt}$\\
\>\>\>   $k = k+1 \rule{0pt}{12pt}$\\
\>\>\>   $q_{k}^{(\skew)} = \:(e_{j}\T e_{i} - e_{i} \T e_{j})/\sqrt{2} \rule{0pt}{13pt}$\\
\>\> {\bf end} $\rule{0pt}{0pt}$\\
\>{\bf end} \\
\end{code}
\stepcounter{equation}
Since {\tt reshape}$(e_{i}\T e_{j},n,n) = e_{j}e_{i}^{T}$, it is clear that the columns of $Q^{(\sym)}$ reshape
to symmetric matrices while the columns of $Q^{(\skew)}$ reshape
to skew-symmetric matrices. Define the $n^{2}$-by-$n^{2}$ matrix 
\begin{equation}
\label{Qdef}
Q_{\nn} \:=\: \left[ Q_{\nn}^{(\sym)} \:|\: Q_{\nn}^{(\skew)} \:\right]
\end{equation}
e.g.,
\[
\!\!\!
Q_{\mbox{\tiny 33}} \:=\:
\left[
\begin{array}{cccccc|rrr}
1 & 0 & 0 & 0 & 0 & 0 & 0 & 0  & 0 \rule{0pt}{11pt}\\
0 & \alpha & 0 & 0 & 0 & 0 & \alpha & 0  & 0 \rule{0pt}{11pt}\\
0 & 0 & \alpha & 0 & 0 & 0 & 0 & \alpha  & 0 \rule{0pt}{11pt}\\
0 & \alpha & 0 & 0 & 0 & 0 & -\alpha & 0  & 0 \rule{0pt}{11pt}\\
0 & 0 & 0 & 1 & 0 & 0 & 0 & 0  & 0 \rule{0pt}{11pt}\\
0 & 0 & 0 & 0 & \alpha & 0 & 0 & 0  & \alpha \rule{0pt}{11pt}\\
0 & 0 & \alpha & 0 & 0 & 0 & 0 & -\alpha  & 0 \rule{0pt}{11pt}\\
0 & 0 & 0 & 0 & \alpha & 0 & 0 & 0  & -\alpha \rule{0pt}{11pt}\\
0 & 0 & 0 & 0 & 0 & 1 & 0 & 0  & 0 \rule{0pt}{11pt}\\
\end{array}
\right],  \qquad \alpha = \frac{1}{\sqrt{2}}.
\]
It is clear that this matrix is orthogonal. Here is a formal proof  together with a verification that
$Q_{\nn}$  block diagonalizes a matrix with PS-symmetry.

\begin{theorem}
If \inm{A}{n^{2}}{n^{2}} is PS-symmetric and $Q_{\nn}$ is defined by (2.6), then $Q_{\nn}$ is orthogonal and
\begin{equation}
Q_{\nn}^{T}AQ_{\nn} 
\:=\:
\left[ \begin{array}{cc} A^{(\mbox{\rm \tiny sym})} & 0 \\ 0 & A^{(\mbox{\rm \tiny skew})} \rule{0pt}{13pt}
\end{array}\right]
\end{equation}
where \inm{A^{(\mbox{\rm \tiny sym})}}{n_{\mbox{\rm \tiny sym}}}{n_{\mbox{\rm \tiny sym}}} and
\inm{A^{(\mbox{\rm \tiny skew})}}{n_{\mbox{\rm \tiny skew}}}{n_{\mbox{\rm \tiny skew}}}.
\end{theorem}
\begin{proof}
All the columns in $Q_{\nn}$ have unit length so the problem is to establish that any pair of its columns are orthogonal to each other.
It is obvious that $\{ e_{1}\T e_{1},\ldots,e_{n}\T e_{n}\}$ is an orthonormal set of vectors and that
\[
(e_{i} \T e_{i})^{T}(e_{p} \T e_{q}) = (e_{i}^{T}e_{p})(e_{i}^{T}e_{q}) = 0 
\]
provided $p\neq q$. It follows that
any column of the form $e_{i} \T e_{i}$ is orthogonal to all the other columns in $Q_{\nn}$. 
Using the Kronecker delta $\delta_{ij}$, 
if $i\neq j$ and $p\neq q$, then
\[
(e_{i} \T e_{j} + e_{j} \T e_{i})^{T})(e_{p} \T e_{q} - e_{q} \T e_{p}) \:=\: \delta_{ip}\delta_{jq} + \delta_{iq}\delta_{jp} - \delta_{iq}\delta_{jp} - \delta_{jq}\delta_{ip} \:=\: 0.
\]
This confirms that  
\begin{equation}
\left[Q_{\nn}^{(\skew)}\right]^{T}Q_{\nn}^{(\sym)} = 0.
\end{equation}
If $(i,j)$, $(j,i)$, $(p,r)$ and $(r,p)$ are distinct index pairs, then it is easy to show that

\begin{eqnarray*}
(e_{i} \T e_{j} + e_{j} \T e_{i})^{T})(e_{p} \T e_{q} + e_{q} \T e_{p}) &=& \delta_{ip}\delta_{jq} + \delta_{iq}\delta_{jp} + \delta_{iq}\delta_{jp} + \delta_{jq}\delta_{ip} \:=\: 0\\
(e_{i} \T e_{j} - e_{j} \T e_{i})^{T})(e_{p} \T e_{q} - e_{q} \T e_{p}) &=& \delta_{ip}\delta_{jq} - \delta_{iq}\delta_{jp} - \delta_{iq}\delta_{jp} + \delta_{jq}\delta_{ip} \:=\: 0\rule{0pt}{15pt}.
\end{eqnarray*}

\medskip

\noindent
These equations establish  that the columns of both $Q_{\nn}^{(\sym)}$ and $Q_{\nn}^{(\skew)}$ are orthonormal. Combined with (2.9) we see that
$Q_{\nn}$ is an orthogonal matrix. 

To confirm that this matrix block diagonalizes a PS-symmetric $A$ we   observe using (2.7) that
\[
Q_{\nn}^{T} A Q_{\nn}\:=\:
\left[ 
\begin{array}{cc}
 \left[\:Q_{\nn}^{(\sym)} \right]^{T}  A\: Q_{\nn}^{(\sym)} &  \left[\:Q_{\nn}^{(\sym)} \right]^{T}  A \:Q_{\nn}^{(\skew)} \\
 \,\left[\:Q_{\nn}^{(\skew)}\right]^{T}A \:Q_{\nn}^{(\sym)}& \left[\:Q_{\nn}^{(\skew)}\right]^{T}  A\: Q_{\nn}^{(\skew)}  \rule{0pt}{20pt}
\end{array}
\right].
\]

\bigskip

\noindent
Since  $\Pi_{\nn} Q_{\nn}^{(\sym)} = Q_{\nn}^{(\sym)}$
and
$\Pi_{\nn} Q_{\nn}^{(\skew)} = -Q_{\nn}^{(\skew)}$
it follows that
\[
[Q_{\nn}^{(\sym)} ]^{T}  A Q_{\nn}^{(\skew)} \:=\:
[Q_{\nn}^{(\sym)} ]^{T}  \Pi_{\nn} A \Pi_{\nn} Q_{\nn}^{(\skew)}
\:=\: -[Q_{\nn}^{(\sym)} ]^{T}  A Q_{\nn}^{(\skew)} \:=\: 0.  \rule{0pt}{18pt}
\]

\medskip

\noindent
Setting
\begin{eqnarray}
A^{(\sym)} &=& Q_{\nn}^{(\sym)}\rule{0pt}{5pt}^{T}\,A\,Q_{\nn}^{(\sym)}\\
A^{(\skew)} &=& Q_{\nn}^{(\skew)}\rule{0pt}{5pt}^{T}\,A\,Q_{\nn}^{(\skew)}\rule{0pt}{14pt}
\end{eqnarray}
completes the proof of the theorem.
\end{proof}

\bigskip

The efficient formation of of $A^{(\sym)}$ and $A^{(\skew)}$ is critical to our method and to that end we develop characterization of these blocks that is much more useful than (2.10) and (2.11).
Define the index vectors  \symn$\in \R^{n_{\sym}}$ and \skewn$\in \R^{n_{\skew}}$
as follows:
\begin{code}
\\
\> $k=0$\\
\> {\bf for} $j=1:n$\\
\>\> {\bf for} $i=j:n$\\
\>\>\> $k = k+1$\\
\>\>\> \symn$(k) = i + (j-1)n$\\
\>\> {\bf end}\\
\> {\bf end}\\
\> $k=0$ \`(2.12)\\
\> {\bf for} $j=1:n$\\
\>\> {\bf for} $i=j+1:n$\\
\>\>\> $k = k+1$\\
\>\>\> \skewn$(k) = i + (j-1)n$\\
\>\> {\bf end}\\
\> {\bf end}\\
\end{code}
\stepcounter{equation}
If \inm{M}{n}{n} and $v = \mbox{\tt vec}(M)$, then $v(\mbox{\tt sym}_{n})$ is the vector of $M$'s lower triangular entries
and $v(\mbox{\tt skew}_{n})$ is the vector of $M$'s strictly lower triangular entries. (Consider the example
$\mbox{\tt sym}_{3} = [1\; 2 \; 3 \;5 \; 6 \;9]$  and $\mbox{\tt skew}_{3} = [2 \; 3 \; 6]$. ) 
Since
\[
\Pi_{\nn} \,( e_{i} \T e_{j}) \:=\: (e_{j} \T e_{i}),
\]
it  follows from (2.6) that if 
\begin{equation}
T^{(\sym)} \:=\: \frac{I_{n^{2}} + \Pi_{\nn}}{2},
\end{equation}
 then $q_{k}^{(\sym)}$ 
is a multiple of $T^{(\sym)}(:,\symn(k))$ while if
\begin{equation}
T^{(\skew)} \:=\: \frac{I_{n^{2}} - \Pi_{\nn}}{2},
\end{equation}
then $q_{k}^{(\skew)}$ 
is a multiple of $T^{(\skew)}(:,\skewn(k))$.
Indeed, if
the $n^{2}$-by-$n^{2}$ diagonal matrix $\Delta^{(\sym)}$ is defined by
\begin{equation}
\Delta^{(\sym)}_{i+(j-1)n,i+(j-1)n} \:=\: \left\{ \begin{array}{ll} \sqrt{2} & i\neq j\\
1 & i=j \rule{0pt}{14pt}\end{array} \right.
\end{equation}
where $i$ and $j$ each range from 1 to $n$,
then it is easy to verify that the columns of $T^{(\sym)}\Delta^{(\sym)}$ have unit 2-norm and
\begin{equation}
Q_{\nn}^{(\sym)} \:=\: T^{(\sym)}(:,u) \cdot \Delta^{(\sym)}(u,u), \qquad u = \symn. \rule[-10pt]{0pt}{25pt}
\end{equation}
The scaling to obtain  $Q_{\nn}^{(\skew)}$ is simpler:
\begin{equation}
Q_{\nn}^{(\skew)} \:=\: \sqrt{2} \cdot T^{(\skew)}(:,v),\qquad v = \skewn. \rule[-10pt]{0pt}{25pt}
\end{equation}
Note that $T^{(\sym)}$ is symmetric and $T^{(\sym)}T^{(\sym)}  = T^{(\sym)}$. Since $x \in \mathbf{S}_{\nn}^{(\sym)}$ implies
$T^{(\sym)}x = x$, it follows that $T^{(\sym)}$ is the orthogonal projector associated with $\mathbf{S}_{\nn}^{(\sym)}$. Likewise, $T^{(\skew)}$ is the orthogonal projector associated with $\mathbf{S}_{\nn}^{(\skew)}$.

Since $\Pi_{\nn}A\Pi_{\nn} = A$, it is easy to show that
\begin{eqnarray*}
T^{(\sym)}\rule{0pt}{5pt}^{T}\,A\,T^{(\sym)} &=& (I_{n^{2}} + \Pi_{\nn})A(I_{n^{2}} + \Pi_{\nn}) \:=\: (A + A\Pi_{\nn})/2
\rule{0pt}{18pt}\\
T^{(\skew)}\rule{0pt}{5pt}^{T}\,A\,T^{(\skew)} &=& (I_{n^{2}} - \Pi_{\nn})A(I_{n^{2}} - \Pi_{\nn}) \:=\: (A - A\Pi_{\nn})/2
\rule{0pt}{18pt}
\end{eqnarray*}
When these equations are combined with (2.10), (2.11), (2.16), and (2.17)
 we obtain
\begin{eqnarray*}
A^{(\sym)} &\;=\;&  \Delta^{(\sym)}(u,u)\cdot \,\frac{A(u,u)+A(u,:)\Pi_{\nn}(:,u)}{2} \,\cdot \Delta^{(\sym)}(u,u) \rule{0pt}{25pt}\\
A^{(\skew)}&\;=\;& (A(v,v)-A(v,:)\Pi_{\nn}(:,v)) \rule{0pt}{25pt}
\end{eqnarray*}
This can be rewritten as
\begin{eqnarray}
A^{(\sym)} &\;=\;&  \Delta^{(\sym)}(u,u)\cdot \,\frac{A(u,u)+A(u,p(u))}{2} \,\cdot \Delta^{(\sym)}(u,u) \rule{0pt}{25pt}\\
A^{(\skew)}&\;=\;& A(v,v)-A(v,p(v)) \rule{0pt}{25pt}
\end{eqnarray}
where $p = \left[\; 1:n:n^{2} \; 2:n:n^{2} \cdots n:n:n^{2} \:\right]$
is the index vector that defines $\Pi_{\nn}$, i.e., $\Pi_{\nn} = I_{n^{2}}(:,p)$. See (1.18).

\subsection{The Schur Decomposition and SVD of a PS-Symmetric Matrix} 
It is not a surprise
that the Schur decomposition of a PS-symmetric matrix involves a highly structured eigenvector
matrix. If 
\[
[U^{(\sym)}]^{T}A^{(\sym)}U^{(\sym)} \:=\: \Lambda^{(\sym)} \:=\: 
\mbox{diag}(\lambda^{(\sym)}_{1},\ldots,\lambda^{(\sym)}_{n_{\sym}})
\]
and
\[
[U^{(\skew)}]^{T}A^{(\skew)}U^{(\skew)} \:=\: \Lambda^{(\skew)} \:=\: 
\mbox{diag}(\lambda^{(\skew)}_{1},\ldots,\lambda^{(\skew)}_{n_{\skew}})
\]
are the Schur decompositions of the diagonal blocks in (2.8) and the orthogonal matrix $Q$ is defined by
\begin{equation}
Q \:=\:  Q_{\nn} \left[ \begin{array}{cc} U^{(\sym)} & 0 \\ 0 & U^{(\skew)} \rule{0pt}{15pt} \end{array} \right]
\:=\: \left[ \begin{array}{c|c}Q_{\nn}^{(\sym)}U^{(\sym)} & Q_{\nn}^{(\skew)}U^{(\skew)}\end{array}\right]
\end{equation}
then
\begin{equation}
Q^{T}AQ \:=\: \left[ \begin{array}{cc} D^{(\sym)} & 0 \\ 0 & D^{(\skew)} \rule{0pt}{15pt} \end{array} \right]
\:=\:
\mbox{diag}(\lambda^{(\sym)}_{1},\ldots,\lambda_{n_{\sym}}^{(\sym)},\lambda^{(\skew)}_{1},\ldots,\lambda_{n_{\skew}}^{(\skew)}) .
\end{equation}
By virtue of how we defined $Q_{\nn}$ in (2.6),  the columns of $Q_{\nn}^{(\sym)}U^{(\sym)}$ (the
``sym-eigenvectors'') reshape to $n$-by-$n$ symmetric matrices. Likewise, the
the columns of $Q_{\nn}^{(\skew)}U^{(\skew)}$ (the
``skew-eigenvectors'') reshape to $n$-by-$n$ skew-symmetric matrices.

Note that this structured Schur decomposition is an unnormalized SVD of $A$. The singular values of $A$ are
the absolute values of the $\lambda$'s. Reordering together with some  ``minus one'' scalings can turn equations
(2.20) and (2.21)
 into a normalized SVD.

\subsection{The Kronecker Product SVD of a PS-Symmetric matrix}

A block matrix with uniformly sized blocks has a {\em Kronecker Product SVD} (KPSVD), see \cite[p.712--14]{gvl4}.
For example, if $A$ an $n$-by-$n$ is a block matrix with $n$-by-$n$ blocks, then there exist
$n$-by-$n$ matrices $B_{1},\ldots,B_{n^{2}}$, and $C_{1},\ldots,C_{n^{2}}$ and scalars 
$\sigma_{1}\geq \cdots \geq \sigma_{n^{2}}\geq 0$ such that
\[
A \:=\: \sum_{k=1}^{n^{2}} \sigma^{(k)} \, (B_{k} \T C_{k}).
\]
The decomposition is  related to the SVD of the $n^{2}$-by-$n^{2}$ matrix $\tilde{A}$ defined by
\begin{equation}
\tilde{A}(i_{2}+(j_{2}-1)n,i_{1}+(j_{1}-1)n)\:=\: A(i_{1}+(i_{2}-1)n,j_{1}+(j_{2}-1)n)
\end{equation}
where the indices $i_{1}$, $i_{2}$, $j_{1}$, and $j_{2}$ range from 1 to $n$.
In particular, if
\[
\tilde{A} \:=\: \sum_{k=1}^{n^{2}} \sigma_{k}\,  b_{k} c_{k}^{T}
\]
is the rank-1 SVD expansion of $\tilde{A}$, then 
\begin{eqnarray}
B_{k} &=& \mbox{\tt reshape}(b_{k},n,n)\\
C_{k} &=& \mbox{\tt reshape}(c_{k},n,n) \rule{0pt}{14pt}
\end{eqnarray}
for $k=1:n^{2}$. 

We show  that  KPSVD of a  PS-symmetric  matrix is highly structured.
To begin with,  the matrix  $\tilde{A}$ defined by (2.22) is PS-symmetric. Indeed by combining (2.2) and (2.22)
we see that
\begin{equation}
\begin{array}{cl}
  & \tilde{A}(i_{2}+(j_{2}-1)n,i_{1}+(j_{1}-1)n) \\ 
= & \tilde{A}(j_{2}+(i_{2}-1)n,j_{1}+(i_{1}-1)n)  \rule{0pt}{13pt} \\
= & \tilde{A}(i_{1}+(j_{1}-1)n,i_{2}+(j_{2}-1)n) \rule{0pt}{13pt}\\
= & \tilde{A}(j_{1}+(i_{1}-1)n,j_{2}+(i_{2}-1)n) .\rule{0pt}{13pt}
\end{array}
\end{equation}
These equalities show that $\tilde{A}^{T} = \tilde{A}$ and $\tilde{A} = \Pi_{\nn}\tilde{A}\Pi_{\nn}$.
In other words, $\tilde{A}$ is PS-symmetric. From (2.18) and (2.19)  we know that $\tilde{A}$ has a 
rank-1 Schur decomposition expansion of the form
\[
\tilde{A} \:=\: \sum_{i=1}^{n_{\sym}} \lambda^{(\sym)}_{i} \, b^{(\sym)}_{i}\,[b^{(\sym)}_{i}]^{T}
\:+\:
\sum_{i=1}^{n_{\skew}} \lambda^{(\skew)}_{i} \, b^{(\skew)}_{i}\,[b^{(\skew)}_{i}]^{T}
\]
where $\Pi_{\nn}b^{(\sym)}_{i} = b^{(\sym)}_{i}$ and $\Pi_{\nn}b^{(\skew)}_{i} = -b^{(\skew)}_{i}$. We may assume
\[
|\lambda^{(\sym)}_{1}|\:\geq \cdots \geq \: |\lambda^{(\sym)}_{n_{\sym}}|
\]
and
\[
|\lambda^{(\skew)}_{1}|\:\geq \cdots \geq \: |\lambda^{(\skew)}_{n_{\skew}}|.
\]
To get an unnormalized KPSVD of $A$, we follow (2.23) and (2.24) and reshape the eigenvectors  of $\tilde{A}$ into $n$-by-$n$ matrices.
The sym-eigenvectors give us symmetric matrices
$B^{(\sym)}_{1},\ldots,B^{(\sym)}_{n_{\sym}}$
while the skew-eigenvectors give us  skew-symmetric matrices
$B^{(\skew)}_{1},\ldots,B^{(\skew)}_{n_{\skew}}$. Overall
 we obtain
\[
A \:=\: \sum_{i=1}^{n_{\sym}} \lambda^{(\sym)}_{i} \, (B^{(\sym)}_{i} \T B^{(\sym)}_{i})
\:+\:
\sum_{i=1}^{n_{\skew}} \lambda^{(\skew)}_{i} \, (B^{(\skew)}_{i}\T B^{(\skew)}_{i})
\]
which can be regarded as an unnormalized KPSVD of $A$.

\subsection{A Structured Cholesky-Based Representation}

Now assume that $A$  is PS-symmetric \underline{and} positive semidefinite with rank $r$. 
Analogous to how we proceeded in the centrosymmetric case, we  develop
a structured representation of $A$ that is based on pivoted Cholesky factorizations of the matrices
$A^{(\sym)}$ and $A^{(\skew)}$ in (2.18) and (2.19). 
We  compute the pivoted Cholesky factorizations

\begin{eqnarray}
P^{(\sym)}A^{(\sym)}P^{(\sym)}\rule{0pt}{4pt}^{T} &=& L^{(\sym)}{L^{(\sym)}}\rule{0pt}{4pt}^{T}  \\
P^{(\skew)}A^{(\skew)}{P^{(\skew)}}\rule{0pt}{4pt}^{T} &= & L^{(\skew)}{L^{(\skew)}}\rule{0pt}{4pt}^{T} 
\rule{0pt}{16pt}
\end{eqnarray}
where
\begin{eqnarray}
\inm{L^{(\sym)}}{n_{\sym}}{r_{\sym}}, &\qquad  & r_{\sym} = \mbox{rank}(A^{(\sym)})\\
\inm{L^{(\skew)}}{n_{\skew}}{r_{\skew}}, & & r_{\skew} = \mbox{rank}(A^{(\skew)}) \rule{0pt}{15pt}.
\end{eqnarray}

\medskip

\noindent
The matrices $\{L^{(\sym)},P^{(\sym)},L^{(\skew)},P^{(\skew)}\}$ collectively
define a structured representation of $A$, for if
\[
\begin{array}{rcccl}
Y^{(\sym)} \!&\!=\!& Q^{(\sym)}P^{(\sym)}\rule{0pt}{4pt}^{T}L^{(\sym)} &\!=\!& [ \:y_{1}^{(\sym)} \:| \: \cdots \:|\:y^{(\sym)}_{r_{\sym})} \:]\\
Y^{(\skew)} \!&\!=\!& Q^{(\skew)}P^{(\skew)}\rule{0pt}{4pt}^{T}L^{(\skew)} &\!=\!& [ \:y_{1}^{(\skew)} \:| \: \cdots \:|\:y^{(\skew)}_{r_{\skew}} \:], \rule{0pt}{18pt}
\end{array}
\]
then it follows from $A = Q^{(\sym)}A^{(\sym)}{Q^{(\sym)}}^{T}\:+\:Q^{(\skew)}A^{(\skew)}{Q^{(\skew)}}^{T} $ that

\begin{equation}
A \:=\: \sum_{i=1}^{r_{\sym}} y^{(\sym)}_{i}\,[y^{(\sym)}_{i}]\rule{0pt}{4pt}^{T}
\:+\: \sum_{i=1}^{r_{\skew}} y^{(\skew)}_{i}\,[y^{(\skew)}_{i}]^{T}. \rule{0pt}{18pt}
\end{equation}

\medskip

\noindent
Each of the rank-1 matrices in this expansion is
PS-symmetric 
because 

\[
\begin{array}{rclcccl}
\Pi_{\nn}Y^{(\sym)} \!\!\!&\!=\!&\! (\Pi_{\nn}Q^{(\sym)})(P^{(\sym)}\rule{0pt}{4pt}^{T}L^{(\sym)}) \!\!\!&\!=\!&\! 
Q^{(\sym)}(P^{(\sym)}\rule{0pt}{4pt}^{T}L^{(\sym)})\!\!\!&\!=\!&\! Y^{(\sym)}\\
\Pi_{\nn}Y^{(\skew)} \!\!\!&\!=\!&\! (\Pi_{\nn}Q^{(\skew)})(P^{(\skew)}\rule{0pt}{4pt}^{T}L^{(\skew)}) \!\!\!&\!=\!&\! 
-Q^{(\sym)}(P^{(\skew)}\rule{0pt}{4pt}^{T}L^{(\skew)})\!\!\!&\!=\!&\! -Y^{(\skew)} .\rule{0pt}{18pt}
\end{array}
\]

\medskip
\noindent
Thus, by combining the block diagonalization  with pivoted Cholesky factorizations  we can efficiently represent 
a given positive semidefinite matrix with PS-symmetry. Here is a summary of the procedure:
\begin{figure}[h]
\begin{center}
\begin{tabular}{|ll|}
\hline
\multicolumn{2}{|c|}{Representing a Positive Semidefinite PS-Symmetric $A\rule[-5pt]{0pt}{16pt}$}\\
\hline
1. & Form $A^{(\sym)}$ using (2.18).\rule{0pt}{14pt}\\
2. & Compute the pivoted Cholesky factorization of $A^{(\sym)}$. See (2.26) and (2.28). \rule{0pt}{14pt}\\
3. & Form $A^{(\skew)}$ using (2.19) \rule{0pt}{14pt}\\
4. & Compute the pivoted Cholesky factorization of $A^{(\skew)}$. See (2.27) and (2.29). \rule[-6pt]{0pt}{18pt}\\
\hline
\end{tabular}
\end{center}
\caption{Computing the representation $\{L^{(\sym)},P^{(\sym)},L^{(\skew)},P^{(\skew)}\}$of a PS-Symmetric Matrix}
\end{figure}

\bigskip

\noindent
By truncating  the summations in (2.30) we can use this framework to construct low-rank approximations that are also PS-symmetric.
We shall have more to say about this and related implementation issues in \S4. To anticipate the discussion we share 
some benchmarks in {\sc Fig. 2.2} 
\begin{figure}[h]
\begin{center}
\small

\bigskip

\begin{tabular}{|c||c|c||c|c|}
\hline
& \multicolumn{2}{|c||}{$r_{\mbox{\tiny sym}}=n_{\mbox{\tiny sym}} \:,\;r_{\mbox{\tiny skew}}=n_{\mbox{\tiny skew}} \rule[-5pt]{0pt}{14pt}$} &
\multicolumn{2}{|c|}{$r_{\mbox{\tiny sym}}=n\:,\;r_{\mbox{\tiny skew}} = n$}\\
\hline
$n$ & $T_{u}/T_{s}$ & $T_{\mbox{\tiny \em set-up}}/T_{s}$& $T_{u}/T_{s}$ &$T_{\mbox{\tiny \em set-up}}/T_{s}$ \rule[-5pt]{0pt}{14pt}\\
\hline
39 & 1.69 & 0.44 & 0.65 & 0.75  \rule[-4pt]{0pt}{14pt}\\
55 & 2.33 & 0.32 & 0.69 & 0.78  \rule[-4pt]{0pt}{14pt}\\
67 & 2.48 & 0.28 & 0.67 & 0.69  \rule[-4pt]{0pt}{14pt}\\
77 & 2.81 & 0.22 & 0.75 & 0.69  \rule[-4pt]{0pt}{14pt}\\
\hline
\end{tabular}
\end{center}
\caption{$T_{u}$ is the time required to compute the Cholesky factorization of $A$, $T_{s}$ is the
time required to set up $A^{(\mbox{\tiny sym})}$ and $A^{(\mbox{\tiny skew})}$ and compute their Cholesky factorizations, and $T_{\mbox{\tiny \em set-up}}$ is the time required to just set-up $A^{(\mbox{\tiny sym})}$ and $A^{(\mbox{\tiny skew})}$. The LAPACK procedures POTRF (unpivoted Cholesky calling level-3 BLAS) and PSTRF (pivoted Cholesky calling level-3 BLAS) were used for full rank and low rank cases respectively. Results are based on running numerous  random trials for each combination of $n$ and $(r_{\mbox{\tiny sym}},r_{\mbox{\tiny skew}})$. A  single core of the Intel(R) Core(TM) i5-3210M CPU @ 2.50GHz was used.}
\end{figure}
The results are similar to what is reported in {\sc Fig 1.2} for the centrosymmetric problem. In the full rank case we anticipate a
four-fold speed-up because the matrices $A^{(\sym)}$ and $A^{(\skew)}$ have dimension that is about half the dimension  of $A$.
However, $T_{u}/T_{s}$ is somewhat  less than 4 because the set-up time fraction
$T_{\mbox{\tiny \em set-up}}/T_{s}$ is nontrivial. In the low-rank case, this overhead rivals the cost of the half-size factorizations
because of the reliance upon traditional right-looking procedures that force us to carry out the complete block diagonalization beforehand.

\section{((1,2),(3,4))-Symmetry}

We now apply the results of the previous section to the structured multilinear product (1.2).
To drive the discussion we consider an example that arises in quantum chemistry and related 
application areas. The underlying tensor is ((1,2),(3,4)))-symmetric and its $[1,2]\times[3,4]$
unfolding is near a matrix with very low rank.

 \subsection{Unfolding a ((1,2),(3,4))-Symmetric Tensor}

If ${\cal A} \in \R^{n\times n \times n \times n}$ is ((1,2),(3,4))-symmetric, then its
$[1,2]\times [3,4]$ unfolding has three important properties that are tabulated in {\sc Fig.}3.1.
\begin{figure}[h]
\begin{center}
\begin{tabular}{|l|c|}
\hline
\multicolumn{1}{|c|}{Symmetry in $\cal A$} & \multicolumn{1}{c|}{Implication for $A = {\cal A}_{\unfold{1}{2}{3}{4}}$}
\rule[-4pt]{0pt}{15pt}\\
\hline
${\cal A}(i_{1},i_{2},i_{3},i_{4}) = {\cal A}(i_{3},i_{4},i_{1},i_{2})$ &  $A = A^{T}$\rule{0pt}{14pt}\\
${\cal A}(i_{1},i_{2},i_{3},i_{4}) = {\cal A}(i_{2},i_{1},i_{3},i_{4})$ &  $\Pi_{\nn}A = A$ \rule{0pt}{14pt}\\
${\cal A}(i_{1},i_{2},i_{3},i_{4}) = {\cal A}(i_{1},i_{2},i_{4},i_{3})$ &  $A\Pi_{\nn} = A$\rule[-6pt]{0pt}{20pt}\\
\hline
\end{tabular}
\end{center}
\caption{Unfolding a ((1,2),(3,4)-Tensor}
\end{figure}
We refer to an $n^{2}$-by-$n^{2}$ matrix $A$ that satisfies $A=A^{T}$, $\Pi_{\nn}A = A$,
and $A = A\Pi_{\nn}$ as a ((1,2),(3,4))-{\em symmetric matrix}.
Such a matrix is also PS-symmetric because
the properties
$\Pi_{\nn}A = A$ and
 $A\Pi_{\nn} = A$ imply $\Pi_{\nn} A \Pi_{\nn} = A$. 
This permits us to say a little more about the block diagonalization in (2.8).
\begin{theorem}
If the $n^{2}$-by-$n^{2}$ matrix $A$ is ((1,2),(3,4))-symmetric,
then
\[
Q_{\nn}^{T} A Q_{\nn} \:=\: \left[ \begin{array}{cc} A^{(sym)} & 0 \\ 0 & 0 \rule{0pt}{15pt}\end{array} \right]
\]
where $Q_{nn}$ is defined by (2.6). In other words, the diagonal block $A^{(\skew)}$ in Theorem 2.1 is zero. Moreover,
\begin{equation}
A^{(\sym)} \;=\;  \Delta^{(\sym)}(u,u)\cdot \,A(u,u) \,\cdot \Delta^{(\sym)}(u,u)
\end{equation}where  $\Delta^{(\sym)}$ is defined by (2.15) and $u = \mbox{\tt sym}_{n}$ is given by (2.12).
\end{theorem}
\begin{proof}
Using (2.11) and the properties
$A\Pi_{\nn} = A$  
and $\Pi_{nn}Q_{\nn}^{(\skew)} = -Q_{\nn}^{(\skew)}\rule{0pt}{20pt}$, we have
\begin{eqnarray*}
A^{(\skew)} &=& 
Q_{\nn}^{(\skew)}\rule{0pt}{5pt}^{T}A Q_{\nn}^{(\skew)}
\:=\:
Q_{\nn}^{(\skew)}\rule{0pt}{5pt}^{T}(A\Pi_{\nn}) Q_{\nn}^{(\skew)}  \rule{0pt}{14pt}\\
&=&
Q_{\nn}^{(\skew)}\rule{0pt}{5pt}^{T}A(\Pi_{\nn}Q_{\nn}^{(\skew)}) \rule{0pt}{14pt} 
\:=\:
-Q_{\nn}^{(\skew)}\rule{0pt}{5pt}^{T}A Q_{\nn}^{(\skew)} \:=\:-A^{(\skew)}\rule{0pt}{14pt}
\end{eqnarray*}
Thus, $A^{(\skew)} = 0$. Equation (3.1)  follows by noting that $A(u,p(u)) = A(u,u)$ in (2.18).
\end{proof}

\bigskip

\noindent
With this added bit of structure we can construct a representation
that is more abbreviated than what is laid out in Figure 2.1 for matrices that are
merely PS-symmetric.
Observe in {\sc Fig 3.2} that only a  single half-sized factorization is required.
\begin{figure}[h]
\begin{center}
\begin{tabular}{|ll|}
\hline
\multicolumn{2}{|c|}{Representing a ((1,2),(3,4))-Symmetric Matrix $A$ that is Positive Semidefinite $\rule[-5pt]{0pt}{16pt}$}\\
\hline
1. & Form $A^{(\sym)}$ using (3.1). \rule{0pt}{14pt}\\
2. & Compute the pivoted Cholesky factorization of $A^{(\sym)}$. See (2.26) and (2.27). \rule{0pt}{14pt}\\
\hline
\end{tabular}
\end{center}
\caption{Computing the representation $\{L^{(\sym)},P^{(\sym)}\}$of a ((1,2),(3,4))-Symmetric Matrix}
\end{figure}
The impact of the set-up overhead in the first step is discussed in \S4.

\subsection{An Example}

The four-index {\em Electron Repulsion Integral} (ERI) tensor ${\cal A}\in \mathbb{R}^{n\times n\times n\times n}$ is defined by 
\begin{equation}
  {\cal A}(i_{1},i_{2},i_{3},i_{4}) 
\;=\; \int_{\mathbb{R}^3} \int_{\mathbb{R}^3} \frac{\phi_{i_{1}}(\textbf{r}_{1}) \phi_{i_{2}}(\textbf{r}_1)\phi_{i_{3}}(\textbf{r}_2)\phi_{i_{4}}(\textbf{r}_2)}{\|\textbf{r}_1-\textbf{r}_2\|} d\textbf{r}_1 d\textbf{r}_2
\label{eq:34}
\end{equation}
where a set of basis functions $\{\phi_k\}_{1\leq k\leq n}$ is given such that $\phi_k\in H^1(\mathbb{R}^3)$. In general, $\phi_k$ are complex basis functions but in this paper we assume real basis functions. The simplest real basis functions $\phi_k$ are Gaussians parametrized by the exponents \inv{\alpha_k}{} and centers \inv{\textbf{r}_k}{3} for $k=1,\dots,n$, e.g.
\[\phi_k(\textbf{r})=g_k(\textbf{r}-\textbf{r}_k)=(2\alpha_k/\pi)^{3/4}e^{-\alpha_k \|\textbf{r}-\textbf{r}_k\|^2}\]
Typically, more sophisticated basis functions are composed from linear combinations of these simple Gaussians \cite{gill}.

The ERI tensor is essential to electronic structure theory and ab initio quantum chemistry. Efficient numerical algorithms for computing and representing this tensor have been a major preoccupation for researchers interested  in ab initio quantum chemistry \cite{BL,HPS,RW,onealChol}.

Notice that for the ERIs,
\[{\cal A}(i_1 , i_2 , i_3 , i_4 ) = {\cal A}(i_3 , i_4 , i_1 , i_2 )\]
because the order of integration does not matter. Also, due to the commutativity of scalar multiplication:
\[{\cal A}(i_1 , i_2 , i_3 , i_4 ) = {\cal A}(i_2, i_1 , i_3 , i_4) = {\cal A}( i_1 , i_2 , i_4 ,i_3)={\cal A}( i_2 , i_1 , i_4 ,i_3)\]

%Notice that in the definition of ${\cal A}$, if we exchange the variables of integration $\textbf{r}_1$ and $\textbf{r}_2$, then the value of the integral is unchanged because the order of the $\mathbf{r}_{1}$ and $\mathbf{r}_{2}$ integrals is immaterial.  Likewise, the value of the integral is unchanged if we interchange $i_{1}$ and $i_{2}$  or if we interchange $i_{3}$ and $i_{4}$.
Thus, if $\cal A$ is the tensor defined by \eqref{eq:34}, then it is ((1,2),(3,4)) symmetric. Moreover,
$A = {\cal A}_{\unfold{1}{2}{3}{4}}$ is positive definite since it is a Gram matrix for the product-basis set $\{\phi_i\phi_j\}$ in the Coulomb metric $\langle\cdot,\frac{1}{\|\textbf{r}_1-\textbf{r}_2\|}\cdot\rangle$.
See \cite{KKS} for details.

\subsection{A Structured Multilinear Product}

A structured version of (3.2) arises in the  Hartree-Fock method, an important technique
for those concerned with the {\em ab initio} calculation of electronic structure.
Szabo and Ostlund \cite{SO} is an excellent general reference in this regard.
For an accurate treatment of electronic correlation effects, it is convenient to transform the ERI tensor from the atomic orbital basis $\{\phi_k(\textbf{r})\}$ to the molecular orbital basis $\{\psi_k(\textbf{r})\}$. The change of basis is defined by
\begin{equation}
  \psi_p \:=\: \sum_{q=1}^n X(p,q)\phi_q \hspace{2em} p=1,2,\dots, n
\end{equation}
where \inm{X}{n}{n} is given.
The goal is  to transform the atomic orbital basis ERI tensor ${\cal A}$ into the following molecular orbital basis ERI tensor ${\cal B}\in \mathbb{R}^{n\times n\times n\times n}$ defined by
\begin{equation}
  {\cal B}(i_{1},i_{2},i_{3},i_{4}) \:=\: \int_{\mathbb{R}^3} \int_{\mathbb{R}^3} \frac{\psi_{i_{1}}(\textbf{r}_1) \psi_{i_{2}}(\textbf{r}_1)\psi_{i_{3}}(\textbf{r}_2)\psi_{i_{4}}(\textbf{r}_2)}{\|\textbf{r}_1-\textbf{r}_2\|} d\textbf{r}_1 d\textbf{r}_2.
\end{equation}
By substituting (3.3) into (3.2) it is easy to show that this tensor is given by

\begin{equation}
 {\cal B}(i_{1},i_{2},i_{3},i_{4}) \;=\!
 \sum_{j_{1},j_{2},j_{3},j_{4}=1}^{n}  \!\!
{\cal A}(j_{1},j_{2},j_{3},j_{4})
X(i_{1},j_{1}) X(i_{2},j_{2})X(i_{3},j_{3})X(i_{4},j_{4}).
\end{equation}

To analyze and exploit the structure of this computation, we start with the fact that it is
a special case of 
the general multilinear product
\begin{equation}
 {\cal B}(j_{1},j_{2},j_{3},j_{4}) =
 \sum_{i_{1},i_{2},i_{3},i_{4}=1}^{n} 
{\cal A}(i_{1},i_{2},i_{3},i_{4})
X_{1}(i_{1},j_{1}) X_{2}(i_{2},j_{2})X_{3}(i_{3},j_{3})X_{3}(i_{4},j_{4}).
\end{equation}
It can be shown that 
\begin{eqnarray}
{\cal B}_{\unfold{1}{2}{3}{4}} &=& (X_{2} \T X_{1})\,{\cal A}_{\unfold{1}{2}{3}{4}}\,(X_{4} \T X_{3})^{T}
\end{eqnarray}
See \cite[p.728-9]{gvl4}. Thus, if
\begin{equation}
 {\cal B}(i_{1},i_{2},i_{3},i_{4}) \;=\!
 \sum_{j_{1},j_{2},j_{3},j_{4}=1}^{n}  \!\!
{\cal A}(j_{1},j_{2},j_{3},j_{4})
X(i_{1},j_{1}) X(i_{2},j_{2})X(i_{3},j_{3})X(i_{4},j_{4}).
\end{equation}
then it follows  that
\[
{\cal B}_{\unfold{1}{2}{3}{4}} \:=\: (X \T X)\,{\cal A}_{\unfold{1}{2}{3}{4}}\,(X \T X)^{T}.
\]
It is easy to verify that if the tensor $\cal A$ is ((1,2),(3,4))-symmetric then the tensor $\cal B$ is  also ((1,2),(3,4))-symmetric.
Indeed,
\begin{eqnarray*}
\Pi_{\nn} {\cal B}_{\unfold{1}{2}{3}{4}} &=& \Pi_{\nn}(X \T X)\,{\cal A}_{\unfold{1}{2}{3}{4}}\,(X \T X)^{T}\\
&=&(X \T X)(\Pi_{\nn}\,{\cal A}_{\unfold{1}{2}{3}{4}})\,(X \T X)^{T}\rule{0pt}{15pt}\\
&=& (X \T X)\,{\cal A}_{\unfold{1}{2}{3}{4}}\,(X \T X)^{T}\:=\:{\cal B}_{\unfold{1}{2}{3}{4}} .
\rule{0pt}{15pt}
\end{eqnarray*}
where we  used the fact that
$\Pi_{\nn}(M_{1} \T M_{2}) = (M_{2} \T M_{1})\Pi_{\nn}$ for all \inm{M_{1},M_{2}}{n}{n}. See \cite[p.27]{gvl4}.
Likewise, 
${\cal B}_{\unfold{1}{2}{3}{4}} \Pi_{\nn} = {\cal B}_{\unfold{1}{2}{3}{4}}. $
Since
${\cal B}_{\unfold{1}{2}{3}{4}} $ is obviously symmetric, we see that this matrix (and hence the tensor $\cal B$) is
 ((1,2),(3,4)) symmetric.

\section{Discussion}

To check out the ideas presented in the previous sections, we implemented the method displayed in {\sc Fig. 3.2}  and tested it on the
low-rank 
((1,2),(3,4))-symmetric matrices that arise from ERI tensor unfoldings.

\subsection{Low Rank}

It is well known in the TEI setting that ${\cal A}_{\unfold{1}{2}{3}{4}}$ is very close to a matrix whose rank in $O(n)$. 
Indeed,
R{\o}eggen and Wisl{\o}ff-Nilssen \cite{RW} show  that 
$\mbox{rank}_{10^{-p}}(A) \:\approx \: pn$
where $\mbox{rank}_{\delta}(A)$ is the number of $A$'s singular values that are greater than $\delta$.
Affirmations of this heuristic can be found in O'Neal and Simons \cite{onealChol}. For insight we  graphically display the
eigenvalue decay for two simple molecules in {\sc Fig 3.3}. {\sc Fig} 3.4 is  a table of ranks for some larger problems.
See \cite{RW} for more details on the low rank structure.
\begin{figure}[h]
\begin{center}
\begin{tabular}{cc}
\small
$\mbox{H}_{2}\mbox{O}\quad (n=12) $ & $\mbox{C}_{2}\mbox{H}_{4}\quad (n =6)$\\
\includegraphics[height=1.5in]{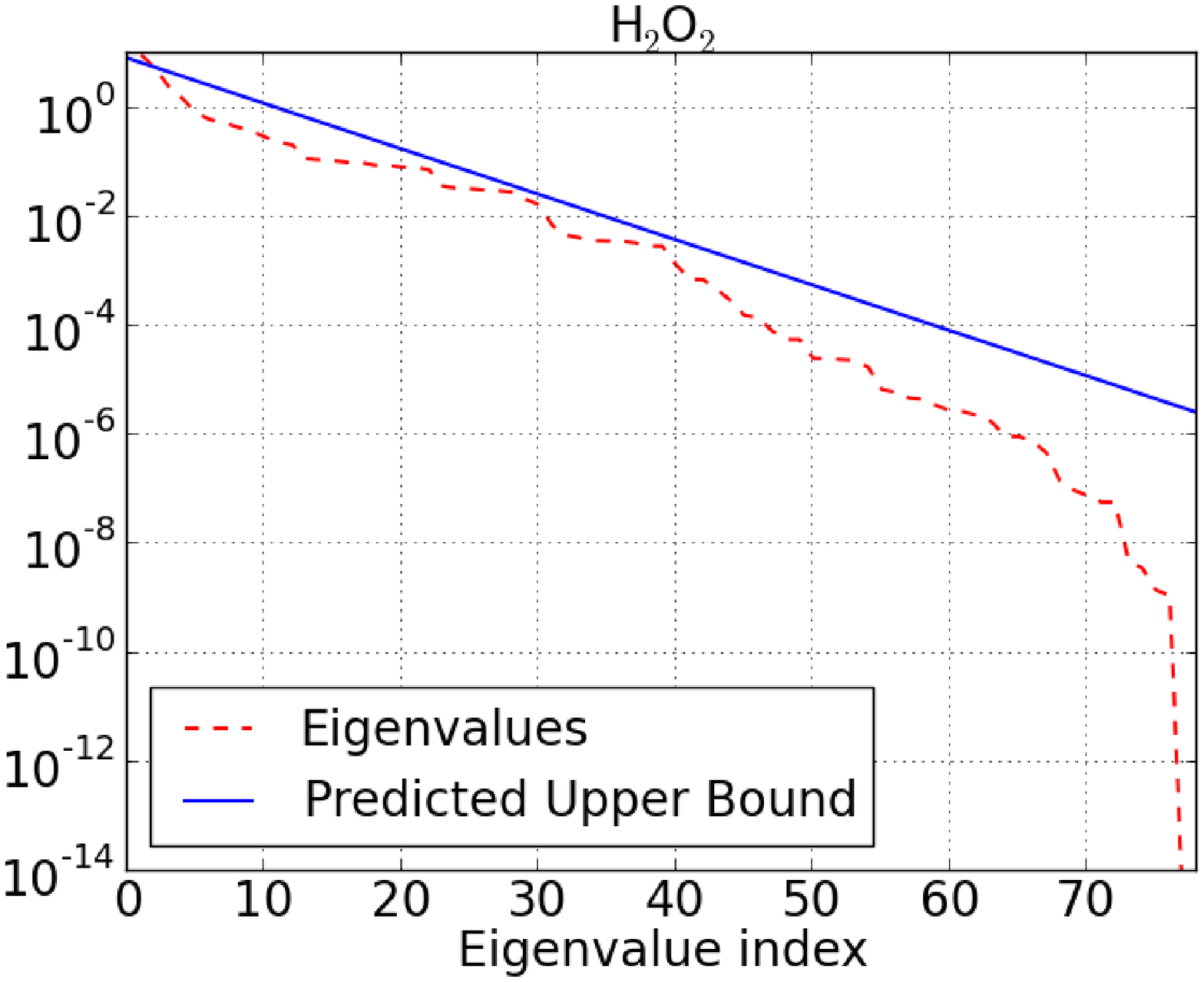}
&
\includegraphics[height=1.5in]{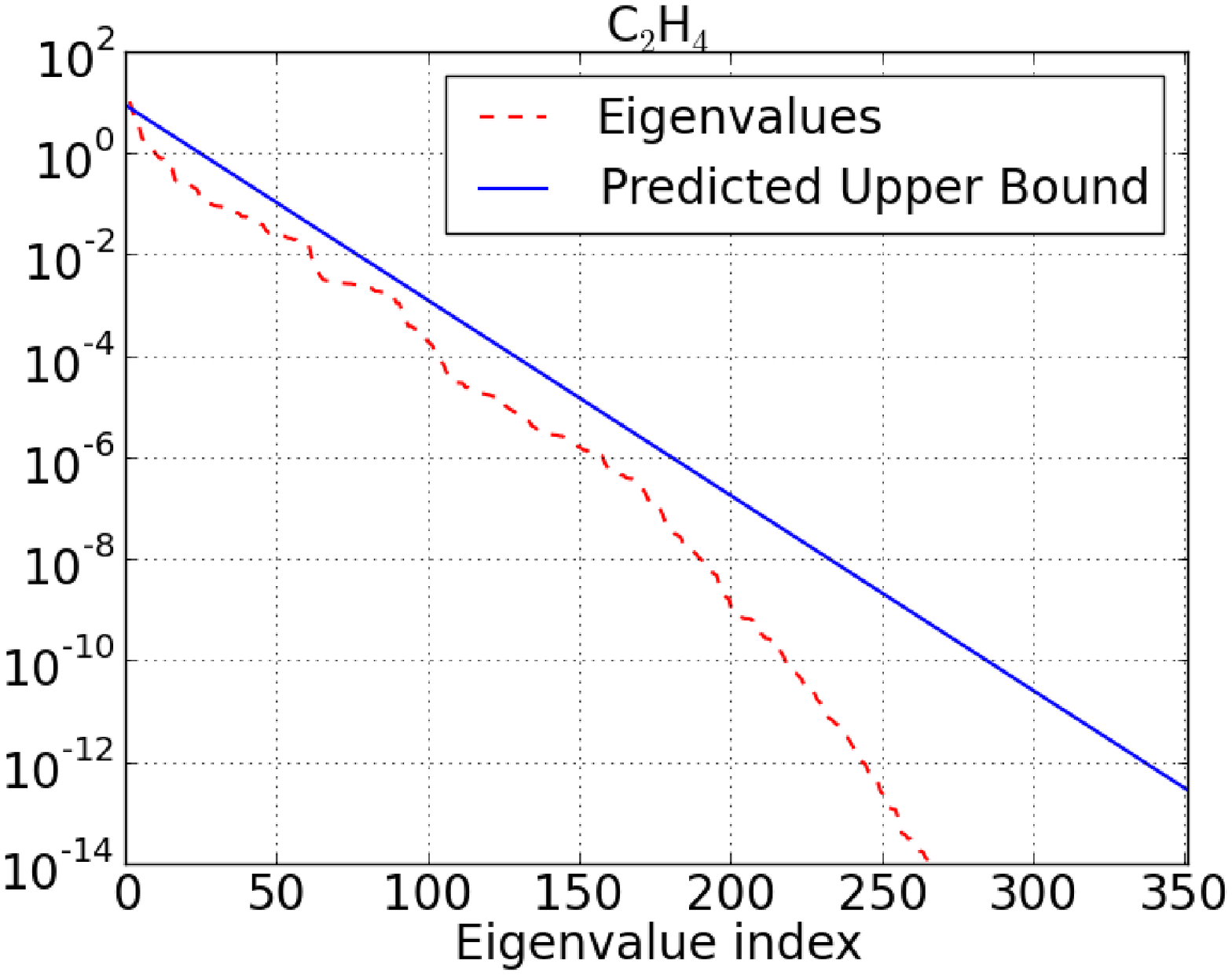}
\end{tabular}
\end{center}
\caption{Eigenvalue decay of ERI matrices generated by the Psi4 Quantum Chemistry Package \cite{psi4}}
\end{figure}
\begin{figure}
\begin{center}
\begin{tabular}{c|rrc}
Molecule & \multicolumn{1}{c}{$n^2$} & \multicolumn{1}{c}{$n$} & \multicolumn{1}{c}{$\mbox{rank}_{10^{-6}}(A)$ }\\
\hline
HF & 1190 &  34 & $\approx 200$  \rule{0pt}{11pt} \\
$\mbox{NH}_{3}$ & 2304 & 48& $\approx  300$   \rule{0pt}{11pt} \\
$\mbox{H}_{2}\mbox{O}_{2}$ & 4624 & 68& $\approx 400$  \rule{0pt}{11pt} \\
$\mbox{N}_2\mbox{H}_4$ & 6724 & 82 & $\approx 500 $  \rule{0pt}{11pt}  \\
$\mbox{C}_2\mbox{H}_5\mbox{OH}$ & 15129 & 123 & $\approx 750 $ \rule{0pt}{11pt} 
\end{tabular}

\caption{Confirmation that $\mbox{rank}_{10^{-6}} \approx 6n$}
\end{center}
\end{figure}

\subsection{A Lazy Evaluation Strategy}

In their highly cited paper  Beebe and Linderberg \cite{BL} demonstrate that by making use of the low rank and positive definiteness of the two-electron integral matrix it is possible to reduce the number of integral evaluations necessary to factorize the matrix, as well as reduce the complexity of a major bottleneck of computational quantum chemistry called the two-electron integral four-index transformation. 
The key idea is to implement the pivoted Cholesky factorization algorithm with lazy evaluation--off-diagonal entries (integrals) are only computed when necessary. To illustrate, after (say) $ k$ steps  of the process on an $N$-by-$N$ matrix $A$,
we have the following partial factorization
\begin{equation}
P_{k}AP_{k}^{T} \:=\: 
\left[ \begin{array}{cc} L_{11} & 0  \\ L_{21} & I_{\smallN-k} \end{array} \right]
\left[ \begin{array}{cc} I_{k} & 0 \\ 0 & \tilde{A} \end{array} \right]
\left[ \begin{array}{cc} L_{11} & 0  \\ L_{21} & I_{\smallN-k} \end{array} \right]^{T}
\end{equation}
Ordinarily, the matrix $\tilde{A}$ is fully available, its diagonal is scanned for the largest entry, and then a
$PAP^{T}$-type of permutation update is performed that brings this largest diagonal entry to the $(k+1,k+1)$ position.
The step is completed by carrying a rank-1 update of the permuted $\tilde{A}$ and this   renders the next column of $L$.
The lazy evaluation version of this recognizes that we do not need the off-diagonal values in $\tilde{A}$ to determine the
pivot. {\em Only the diagonal of  $\tilde{A}$ is necessary to carry out the pivot strategy.} The recipe for the next column 
of $L$ involves (a) previously computed columns of $L$ and (b) entries from that column of $A$ which is associated with the
pivot. It is then an easy matter to update the diagonal of the current $\tilde{A}$ to get the diagonal of the ``next''
$\tilde{A}$. The importance of this lazy-evaluation strategy is that $O(Nk)$ integral evaluations (i.e., $a_{ij}$ evaluations)
are necessary to get through
the $k$-th step.  If the largest diagonal entry in $\tilde{A}$ is less than a small tolerance $\delta$, then because
$\tilde{A}$ is positive definite, 
$\norm{\tilde{A}} = O(\delta)$ and we have the ``right'' to regard $A$ as a rank-$k$ matrix. The overall technique can 
be seen as a combination of Gaxpy-Cholesky, which only needs $A(k:n,k)$ in step $k$ and outer product Cholesky
which is traditionally used in situations that involve diagonal pivoting. 
R{\o}eggen and Wisl{\o}ff-Nilssen \cite{RW} explore the numerical rank of the two-electron integral matrix, and investigate
 the relationship of various thresholds and electronic properties. 
See also  \cite{HPS, KKS}.

While on the subject of lazy evaluation, it is important to stress that the matrix entries in $A^{(\sym)}$ are essentially
entries from $A$. See (3.2). Thus,  when we apply our implementation of pivoted Cholesky to to $A^{(\sym)}$ with lazy evaluation, there
are no extra $a_{ij}$ computations. In other words, our method requires half the work, half the storage, and half the electronic repulsion integrals
as traditional Cholesky-based methods. The table displayed in {\sc Fig 4.3} confirms these observations

\begin{figure}[h]
  \centering
  \begin{tabular}{c|c|c|c|c}
    & $\begin{array}{c}n=44\\r=345\end{array}$ & $\begin{array}{c}n=72\\r=560\end{array}$ & $\begin{array}{c}n=88\\r=720\end{array}$ & $\begin{array}{c}n=116\\r=918\end{array}$ \\
    \hline
    $T_{u}/T_{s}$ & 1.84  &1.90  &1.89 &1.93 \rule{0pt}{14pt}\\
    $S_u/S_s$ & 1.95 & 1.97 & 1.97 & 1.98 \rule{0pt}{14pt}\\
    $E_{u}/E_{s}$ & 1.95  & 1.97  & 1.97 & 1.98 \rule{0pt}{14pt}\\
  \end{tabular}
  
  \caption{$T_{u}$ and $T_{s}$ are the time in seconds to factorize $A$ and $A(u,u)$ respectively; $S_{u}$ and $S_{s}$ are the number of bytes allocated to factorize $A$ and $A(u,u)$ respectively; $E_{u}$ and $E_{s}$ are the number of ERI evaluations to factorize $A$ and $A(u,u)$ respectively. Results are based on running Psi4 Lazy Evaluation pivoted Cholesky on the ERI matrix of four different molecules on a single core of a laptop Intel(R) Core(TM) i5-3210M CPU @ 2.50GHz.}
 
\end{figure}
\subsection{Conclusion}
We have used a simple example of multiple symmetries to explore a computational framework that involves
block diagonalization and the pivoted Cholesky factorization. Items on our research agenda include
the extension of these ideas to more intricate forms
of multiple symmetry that arise in higher-order tensor problems and to apply this approach to improve the performance of the Hartree-Fock method in quantum chemistry. Intelligent  data structures and blocking will certainly be part of the picture. Ragnarsson and Van Loan develop a block tensor computation
framework in \cite{RV2}. If multiple symmetries are present, then as in the
matrix case tensions arise between compact storage schemes and ``layout friendly'' matrix multiplication formulations.
See Epifanovsky et al \cite{epifanBlock}, 
and Solomonik, Matthews, Hammond, and Demmel, \cite{demmelCyclops}. 
In \cite{koldaSym} Schatz, Low, van de Geijn, and Kolda   discuss a blocked data structure for symmetric tensors, partial symmetry, and the prospect of building a general purpose library for multi-linear algebra computation. They also discuss a blocking strategy for a symmetric multilinear product. 

\section{Acknowledgments}
This work was partially supported by the Ronald E. McNair Post-baccalaureate Scholar program.

\end{document}